\documentclass[11pt]{amsart}
\usepackage{amsmath}
\usepackage{amsfonts}
\usepackage{amssymb}
\usepackage{enumerate}
\usepackage{graphicx}
\usepackage{hyperref}

\newtheorem{thm}{Theorem}

\newtheorem{cor}[thm]{Corollary}

\newtheorem{remark}[thm]{Remark}
\newtheorem{lemma}[thm]{Lemma}
\newtheorem{prop}[thm]{Proposition}

\newtheorem{exam}[thm]{Example}
\newtheorem{defn}[thm]{Definition}
\newcommand{\bra}[1]{\langle #1 |}
\newcommand{\ket}[1]{| #1 \rangle}
\newcommand{\braket}[2]{\langle #1 | #2 \rangle}
\newcommand{\ketbra}[2]{| #1 \rangle\langle #2 |}
\newcommand{\Tr}{{\rm Tr}}
\newcommand{\bb}[1]{\mathbb{#1}}
\newcommand{\cl}[1]{\mathcal{#1}}

\newcommand{\mnorm}[1]{%
\left\vert\kern-0.9pt\left\vert\kern-0.9pt\left\vert #1
\right\vert\kern-0.9pt\right\vert\kern-0.9pt\right\vert}

\newcommand{\bigmnorm}[1]{%
\big\vert\kern-0.9pt\big\vert\kern-0.9pt\big\vert #1
\big\vert\kern-0.9pt\big\vert\kern-0.9pt\big\vert}

\title[Operator Spaces and Operator Systems in Entanglement Theory]{Minimal and Maximal Operator Spaces and Operator Systems in Entanglement Theory}

\author[N.~Johnston]{Nathaniel Johnston}
\address{Department of Mathematics and Statistics, University of Guelph,
Guelph, Ontario N1G 2W1, Canada}
\email{njohns01@uoguelph.ca}

\author[D.~W.~Kribs]{David W. Kribs}
\address{Department of Mathematics and Statistics, University of Guelph, Guelph, Ontario N1G 2W1, Canada and Institute for Quantum Computing, University of Waterloo, Waterloo, Ontario N2L 3G1, Canada}
\email{dkribs@uoguelph.ca}

\author[V.~I.~Paulsen]{Vern I.~Paulsen}
\address{Department of Mathematics, University of Houston,
Houston, Texas 77204-3476, U.S.A.}
\email{vern@math.uh.edu}

\author[R.~Pereira]{Rajesh Pereira}
\address{Department of Mathematics and Statistics, University of Guelph, Guelph, Ontario N1G 2W1, Canada}
\email{pereirar@uoguelph.ca}

\begin{document}

\begin{abstract}
We examine $k$-minimal and $k$-maximal operator spaces and operator systems, and investigate their relationships with the separability problem in quantum information theory. We show that the matrix norms that define the $k$-minimal operator spaces are equal to a family of norms that have been studied independently as a tool for detecting $k$-positive linear maps and bound entanglement. Similarly, we investigate the $k$-super minimal and $k$-super maximal operator systems that were recently introduced and show that their cones of positive elements are exactly the cones of $k$-block positive operators and (unnormalized) states with Schmidt number no greater than $k$, respectively. We characterize a class of norms on the $k$-super minimal operator systems and show that the completely bounded versions of these norms provide a criterion for testing the Schmidt number of a quantum state that generalizes the recently-developed separability criterion based on trace-contractive maps.\medskip

\end{abstract}

\maketitle

\section{Introduction}

A primary goal of this paper is to formally link central areas of study in operator theory and quantum information theory. More specifically, we connect recent investigations in operator space and operator system theory \cite{Paulsentext,Pistext} on the one hand and the theory of entanglement \cite{HHH09,BZtext} on the other. As benefits of this combined perspective, we obtain new results and new elementary proofs in both areas. We give further details below before proceeding.

Given a (classical description of a) quantum state $\rho$, one of the most basic open questions in quantum information theory asks for an operational criterion for determining whether $\rho$ is separable or entangled. Much progress has been made on this front over the past two decades. For instance, a revealing connection between the separability problem and operator theory was established in \cite{HHH96}, where it was shown that $\rho$ is separable if and only if it remains positive under the application of any positive map to one half of the state. Another more recent approach characterizes separability via maps that are contractive in the trace norm on Hermitian operators \cite{HHH06}. In this work we show that these two approaches to the separability problem can be seen as arising from the theory of minimal and maximal operator systems and operator spaces, respectively.

Additionally, this work can be seen as demonstrating how to rephrase certain positivity questions that are relevant in quantum information theory in terms of norms that are relevant in operator theory instead. For example, instead of using positive maps to detect separability of quantum states, we can construct a natural operator system into which positive maps are \emph{completely} positive. Then the completely bounded norm on that operator system serves as a tool for detecting separability of quantum states as well.

A natural generalization of the characterization of separable states in terms of positive linear maps was implicit in \cite{TH00} and proved in \cite{RA07} -- a state has Schmidt number no greater than $k$ if and only if it remains positive under the application of any $k$-positive map to one half of the state. Recently, a further connection was made between operator theory and quantum information: a map is completely positive on what is known as the maximal (resp. minimal) operator system on $M_n$, the space of $n \times n$ complex matrices, if and only if it is a positive (resp. entanglement-breaking \cite{HSR03}) map \cite{PTT09}. Thus, the maps that serve to detect quantum entanglement are the completely positive maps on the maximal operator system on $M_n$.

Similarly, completely positive maps on ``$k$-super maximal'' and ``$k$-super minimal'' operator systems on $M_n$ \cite{XThesis} have been studied and shown to be the same as $k$-positive and $k$-partially entanglement breaking maps \cite{CK06}, respectively. We will reprove these statements via an elementary proof that shows that the cones of positive elements that define the $k$-super maximal (resp. $k$-super minimal) operator systems are exactly the cones of (unnormalized) states with Schmidt number at most $k$ (resp. $k$-block positive operators).

Analogous to the minimal and maximal operator systems, there are minimal and maximal operator spaces (and appropriate $k$-minimal and $k$-maximal generalizations). We will show the norms that define the $k$-minimal operator spaces on $M_n$ coincide with a family of norms that have recently been studied in quantum information theory \cite{JK10a,JK10b,J10,PPHH10,CKo09,CKS09,DSSTT00} for their applications to the problems of detecting $k$-positive linear maps and NPPT bound entangled states. Furthermore, we will connect the dual of a version of the completely bounded minimal operator space norm to the separability problem and extend recent results about how trace-contractive maps can be used to detect entanglement. We will see that the maps that serve to detect quantum entanglement via norms are roughly the completely contractive maps on the minimal operator space on $M_n$. The natural generalization to norms that detect states with Schmidt number $k$ is proved via a stabilization result for the completely bounded norm from $M_r$ to the $k$-minimal operator space (or system) of $M_n$.

In Section~\ref{sec:QIT} we introduce the reader to the various relevant notions from quantum information theory such as separability and Schmidt rank. In Section~\ref{sec:op_space} we introduce (abstract) operator spaces and the $k$-minimal and $k$-maximal operator space structures, and investigate their relationship with norms that have been used in quantum information theory. In Section~\ref{sec:kMinOpSys} we give a similar treatment to abstract operator systems and the $k$-super minimal and $k$-super maximal operator system structures. We then investigate some norms on the $k$-super minimal operator system structures in Section~\ref{sec:OpSysNorms}. We close in Section~\ref{sec:CBNorm} by considering the completely bounded version of some of the norms that have been presented and establish a relationship with the Schmidt number of quantum states.

\section{Quantum Information Theory Preliminaries}\label{sec:QIT}

Given a vector space $V$, we will use $M_{m,p}(V)$ to denote the space of $m \times p$ matrices with elements from $V$. For brevity we will write $M_m(V) := M_{m,m}(V)$ and $M_m := M_m(\bb{C})$. It will occasionally be convenient to use tensor product notation and identify $M_m \otimes V \cong M_m(V)$ in the standard way, especially when $V = M_n$. We will make use of bra-ket notation from quantum mechanics as follows: we will use ``kets'' $\ket{v} \in \bb{C}^n$ to represent unit (column) vectors and ``bras'' $\bra{v} := \ket{v}^*$ to represent the dual (row) vectors, where $(\cdot)^*$ represents the conjugate transpose. Unit vectors represent pure quantum states (or more specifically, the projection $\ketbra{v}{v}$ onto the vector $\ket{v}$ represents a pure quantum state) and thus we will sometimes refer to unit vectors as states. Mixed quantum states are represented by \emph{density operators} $\rho \in M_m \otimes M_n$ that are positive semidefinite with $\Tr(\rho) = 1$.

A state $\ket{v} \in \bb{C}^m \otimes \bb{C}^n$ is called \emph{separable} if there exist $\ket{v_1} \in \bb{C}^m$, $\ket{v_2} \in \bb{C}^n$ such that $\ket{v} = \ket{v_1} \otimes \ket{v_2}$; otherwise it is said to be \emph{entangled}. The \emph{Schmidt rank} of a state $\ket{v}$, denoted $SR(\ket{v})$, is the least number separable states $\ket{v_i}$ needed to write $\ket{v} = \sum_i \alpha_i \ket{v_i}$, where $\alpha_i$ are some (real) coefficients. The analogue of Schmidt rank for a bipartite mixed state $\rho \in M_m \otimes M_n$ is \emph{Schmidt number} \cite{TH00}, denoted $SN(\rho)$, which is defined to be the least integer $k$ such that $\rho$ can be written in the form $\rho = \sum_i p_i \ketbra{v_i}{v_i}$ with $\{p_i\}$ forming a probability distribution and $SR(\ket{v_i}) \leq k$ for all $i$.

An operator $X = X^* \in M_m \otimes M_n$ is said to be \emph{$k$-block positive} (or a \emph{$k$-entanglement witness}) if $\bra{v}X\ket{v} \geq 0$ for all vectors $\ket{v}$ with $SR(\ket{v}) \leq k$. In the extreme case when $k = \min\{m,n\}$, we see that the $k$-block positive operators are exactly the positive semidefinite operators (since $SR(\ket{v}) \leq \min\{m,n\}$ for all vectors $\ket{v}$), and for smaller $k$ the set of $k$-block positive operators is strictly larger.

In \cite{JK10a,JK10b}, a family of operator norms that have several connections in quantum information theory was investigated. Arising from the Schmidt rank of bipartite pure states, they are defined for operators $X \in M_m \otimes M_n$ as follows:
\begin{align}\label{eq:SchmidtNorm}
	\big\|X\big\|_{S(k)} = \sup_{\ket{v},\ket{w}}\big\{ | \bra{v} X \ket{w} | : SR(\ket{v}),SR(\ket{w}) \leq k \big\}.
\end{align}

\noindent These norms were shown to be useful for determining whether or not an operator is $k$-block positive, and also have applications to the problem of determining whether or not there exist bound entangled non-positive partial transpose states \cite{JK10a,PPHH10}. The problem of computing these norms was investigated in \cite{JK10b}.

The \emph{completely bounded norm} of a linear map $\Phi : M_r \rightarrow M_n$ is defined to be
\begin{align*}
	\big\|\Phi\big\|_{cb} := \sup_{m \geq 1} \Big\{ \big\| (id_m \otimes \Phi)(X) \big\| : X \in M_m(M_r) \text{ with } \big\|X\big\| \leq 1 \Big\}.
\end{align*}

\noindent It was shown by Smith \cite{S83} (and independently later by Kitaev \cite{Kit97} from the dual perspective) that it suffices to fix $m = n$ so that $\big\|\Phi\big\|_{cb} = \big\| id_n \otimes \Phi \big\|$. We will see in Section~\ref{sec:CBNorm} a connection between the norms $\big\| id_k \otimes \Phi \big\|$ for $1 \leq k \leq n$ and the norms~\eqref{eq:SchmidtNorm}.

\section{$k$-Minimal and $k$-Maximal Operator Spaces}\label{sec:op_space}

We will now present (abstract) operator spaces and the $k$-minimal and $k$-maximal operator space structures. An operator space is a vector space $V$ together with a family of $L^\infty$ matrix norms $\|\cdot\|_{M_m(V)}$ on $M_m(V)$ that make $V$ into a matrix normed space. That is, we require that if $A = (a_{ij}),B = (b_{ij}) \in M_{p,m}$ and $X = (x_{ij}) \in M_m(V)$ then $\big\|A \cdot X \cdot B^*\big\|_{M_p(V)} \leq \big\|A\big\|\big\|X\big\|_{M_m(V)}\big\|B\big\|$, where
\begin{align*}
	A \cdot X \cdot B^* := \Big(\sum_{k,\ell = 1}^m a_{ik}x_{k\ell}\overline{b_{j \ell}}\Big) \in M_p(V)
\end{align*}

\noindent and $\big\|A\big\|,\big\|B\big\|$ represent the operator norm on $M_{p,m}$. The $L^\infty$ requirement is that $\big\| X \oplus Y \big\|_{M_{m+p}(V)} = \max\big\{ \|X\|_{M_m(V)},\|Y\|_{M_p(V)}\big\}$ for all $X \in M_m(V), Y \in M_p(V)$.

When the particular operator space structure (i.e., the family of $L^\infty$ matrix norms) on $V$ is not important, we will denote the operator space simply by $V$. We will use $M_n$ itself to denote the ``standard'' operator space structure on $M_n$ that is obtained by associating $M_m(M_n)$ with $M_{mn}$ in the natural way and using the operator norm. For a more detailed introduction to abstract operator spaces, the interested reader is directed to \cite[Chapter 13]{Paulsentext}.

Given an operator space $V$ and a natural number $k$, one can define a new family of norms on $M_m(V)$ that coincide with the matrix norms on $M_m(V)$ for $1 \leq m \leq k$ and are minimal (or maximal) for $m > k$. We will use $MIN^k(V)$ and $MAX^k(V)$ to denote what are called the \emph{$k$-minimal operator space of $V$} and the \emph{$k$-maximal operator space of $V$}, respectively. For $X \in M_m(V)$ we will use $\big\|X\big\|_{M_m(MIN^k(V))}$ and $\big\|X\big\|_{M_m(MAX^k(V))}$ to denote the norms that define the $k$-minimal and $k$-maximal operator spaces of $V$, respectively.

For $X \in M_m(V)$ one can define the $k$-minimal and $k$-maximal operator space norms via
\begin{align}\label{eq:k_min_space}
	\big\|X\big\|_{M_m(MIN^k(V))} & := \sup \Big\{ \big\| (\Phi(X_{ij})) \big\| : \Phi : V \rightarrow M_k \text{ with } \big\| \Phi \big\|_{cb} \leq 1 \Big\} \text{ and} \\ \label{eq:k_max_space}
	\big\|X\big\|_{M_m(MAX^k(V))} & := \sup \Big\{ \big\| (\Phi(X_{ij})) \big\| : \Phi : V \rightarrow \cl{B}(\cl{H}) \text{ with } \big\| id_k \otimes \Phi \big\| \leq 1 \Big\}.
\end{align}

Indeed, the names of these operator space structures come from the facts that if $O(V)$ is any operator space structure on $V$ such that $\|\cdot\|_{M_m(V)} = \|\cdot\|_{M_m(O(V))}$ for $1 \leq m \leq k$ then $\|\cdot\|_{M_m(MIN^k(V))} \leq \mnorm{\cdot}_{M_m(O(V))} \leq \|\cdot\|_{M_m(MAX^k(V))}$ for all $m > k$. In the $k = 1$ case, these operator spaces are exactly the minimal and maximal operator space structures that are fundamental in operator space theory \cite[Chapter 14]{Paulsentext}. The interested reader is directed to \cite{OR04} and the references therein for further properties of $MIN^k(V)$ and $MAX^k(V)$ when $k \geq 2$.

One of the primary reasons for our interest in the $k$-minimal operator spaces is the following result, which says that the $k$-minimal norm on $M_m(M_n)$ is exactly equal to the $S(k)$-norm~\eqref{eq:SchmidtNorm} from quantum information theory.

\begin{thm}\label{thm:MINkChar}
	Let $X \in M_m(M_n)$. Then $\big\|X\big\|_{M_m(MIN^k(M_n))} = \big\|X\big\|_{S(k)}$.
\end{thm}
\begin{proof}
	A fundamental result about completely bounded maps says (see \cite[Theorem 19]{JKP09}, for example) that any completely bounded map $\Phi : M_n \rightarrow M_k$ has a representation of the form
	\begin{align}\label{eq:cbform}
		\Phi(Y) = \sum_{i=1}^{nk} A_i Y B_i^* \quad \text{with } A_i, B_i \in M_{k,n} \, \text{ and } \, \Big\| \sum_{i=1}^{nk}A_i A_i^*\Big\|\Big\| \sum_{i=1}^{nk}B_i B_i^*\Big\| = \big\|\Phi\big\|_{cb}^2.
	\end{align}
		
	\noindent By using the fact that $\Phi$ is completely contractive (so $\big\|\Phi\big\|_{cb} = 1$) and a rescaling of the operators $\big\{A_i\big\}$ and $\big\{B_i\big\}$ we have
	\begin{align*}
		\big\|X\big\|_{M_m(MIN^k(M_n))} = \sup \Big\{ \big\| \sum_{i=1}^{nk} (I_m \otimes A_i) X (I_m \otimes B_i^*) \big\| : \big\| \sum_{i=1}^{nk}A_i A_i^*\big\| = \big\| \sum_{i=1}^{nk}B_i B_i^*\big\| = 1 \Big\},
	\end{align*}
	
	\noindent where the supremum is taken over all families of operators $\big\{ A_i \big\}, \big\{ B_i \big\} \subset M_{k,n}$ satisfying the normalization condition. Now define $\alpha_{ij}\ket{a_{ij}} := A_i^*\ket{j}$ and $\beta_{ij}\ket{b_{ij}} := B_i^*\ket{j}$, and let $\ket{v} = \sum_{j=1}^{k} \gamma_j \ket{c_j} \otimes \ket{j}, \ket{w} = \sum_{j=1}^{k} \delta_j \ket{d_j} \otimes \ket{j} \in \bb{C}^m \otimes \bb{C}^k$ be arbitrary unit vectors. Then simple algebra reveals
	\begin{align*}
		\nu_i\ket{v_i} & := (I_m \otimes A_i^*)\ket{v} = \sum_{j=1}^{k} \alpha_{ij} \gamma_j \ket{c_j} \otimes \ket{a_{ij}} \, \, \text{ and } \\
		\mu_i\ket{w_i} & := (I_m \otimes B_i^*)\ket{w} = \sum_{j=1}^{k} \beta_{ij} \delta_j \ket{d_j} \otimes \ket{b_{ij}}.
	\end{align*}
	
	\noindent In particular, $SR(\ket{v_i}),SR(\ket{w_i}) \leq k$ for all $i$. Furthermore, by the normalization condition on $\big\{A_i\big\}$ and $\big\{B_i\big\}$ we have that
	\begin{align}\label{eq:muNormalize}
		\bra{v}(I_m \otimes \sum_{i=1}^{nk} A_i A_i^*)\ket{v} = \sum_{i=1}^{nk} \nu_i^2 \leq 1 \, \, \, \text{ and } \, \, \, \bra{w}(I_m \otimes \sum_{i=1}^{nk} B_i B_i^*)\ket{w} = \sum_{i=1}^{nk} \mu_i^2 \leq 1.
	\end{align}
	
	\noindent Thus we can write
	\begin{align}\label{eq:cbschmidt}
		\left|\sum_{i=1}^{nk}\bra{v}(I_m \otimes A_i)(X)(I_m \otimes B_i^*)\ket{w}\right| = \left|\sum_{i=1}^{nk} \nu_i \mu_i \bra{v_i} X \ket{w_i}\right| \leq \sum_{i=1}^{nk} \nu_i \mu_i \big| \bra{v_i} X \ket{w_i}\big|.
	\end{align}
	
	\noindent The normalization condition~\eqref{eq:muNormalize} and the Cauchy-Schwarz inequality tell us that there is a particular $i^\prime$ such that the sum~\eqref{eq:cbschmidt} $\leq \left| \bra{v_{i^\prime}} X \ket{w_{i^\prime}}\right|$. Taking the supremum over all vectors $\ket{v}$ and $\ket{w}$ gives the ``$\leq$'' inequality.
	
	The ``$\geq$'' inequality can be seen by noting that if we have two vectors in their Schmidt decompositions $\ket{v} = \sum_{i=1}^k \alpha_i \ket{c_i} \otimes \ket{a_i}$ and $\ket{w} = \sum_{i=1}^k \beta_i \ket{d_i} \otimes \ket{b_i}$, then we can define operators $A,B \in M_{k,n}$ by setting their $i^{th}$ row in the standard basis to be $\bra{a_i}$ and $\bra{b_i}$, respectively. Because the rows of $A$ and $B$ form orthonormal sets, $\big\|A\big\| = \big\|B\big\| = 1$. Additionally, if we define $\ket{v^\prime} = \sum_{i=0}^{k-1} \alpha_i \ket{c_i} \otimes \ket{i}$ and $\ket{w^\prime} = \sum_{i=0}^{k-1} \beta_i \ket{d_i} \otimes \ket{i}$, then
	\begin{align*}
		\big\| (I_m \otimes A)(X)(I_m \otimes B^*) \big\| \geq \big| \bra{v^\prime}(I_m \otimes A)(X)(I_m \otimes B^*)\ket{w^\prime} \big| = \big| \bra{v}X\ket{w} \big|.
	\end{align*}
	
	\noindent Taking the supremum over all vectors $\ket{v},\ket{w}$ with $SR(\ket{v}),SR(\ket{w}) \leq k$ gives the result.
\end{proof}

\begin{remark}{\rm
	When working with an operator system (instead of an operator
        space) $V$, it is more natural to define the norm~\eqref{eq:k_min_space} by taking the supremum over all completely positive unital maps $\Phi: V \rightarrow M_k$ rather than all complete contractions (similarly, to define the norm~\eqref{eq:k_max_space} one would take the supremum over all $k$-positive unital maps rather than $k$-contractive maps). In this case, the $k$-minimal norm no longer coincides with the $S(k)$-norm on $M_m(M_n)$ but rather has the following slightly different form:
	\begin{align}\begin{split}\label{eq:k_min_system}
		\big\|X\big\|_{M_m(OMIN^k(M_n))} = \sup_{\ket{v},\ket{w}}\Big\{ \big| \bra{v} X \ket{w} \big| : & \ SR(\ket{v}),SR(\ket{w}) \leq k \text{ and} \\
			& \ \exists \, P \in M_m \text{ s.t. } (P \otimes I_n)\ket{v} = \ket{w} \Big\},
	\end{split}\end{align}
where the notation $OMIN^k(M_n)$ refers to a new operator system
structure that is being assigned to $M_n,$ which we discuss in detail
in the next section.
	
	\noindent Intuitively, this norm has the same interpretation as the norm~\eqref{eq:SchmidtNorm} except with the added restriction that the vectors $\ket{v}$ and $\ket{w}$ look the same on the second subsystems. We will examine this norm in more detail in Section~\ref{sec:OpSysNorms}. In particular, we will see in Theorem~\ref{thm:k_order_norm} that the norm~\eqref{eq:k_min_system} is a natural norm on the $k$-super minimal operator system structure (to be defined in Section~\ref{sec:kMinOpSys}), which plays an analogous role to the $k$-minimal operator space structure.
}\end{remark}

Now that we have characterized the $k$-minimal norm in a fairly concrete way, we turn our attention to the $k$-maximal norm. The following result is a direct generalization of a corresponding known characterization of the $MAX(V)$ norm \cite[Theorem 14.2]{Paulsentext}.
\begin{thm}\label{thm:kMaxNormChar}
	Let $V$ be an operator space and let $X \in M_m(V)$. Then
	\begin{align*}
		\big\| X \big\|_{M_m(MAX^k(V))} = \inf\Big\{ \big\|A\big\| \big\|B\big\| : & \ A,B \in M_{m,rk}, x_i \in M_k(V), \|x_i\|_{M_k(V)} \leq 1 \text{ with} \\
		& \ X = A\cdot{\rm diag}(x_1,\ldots,x_r)\cdot B^* \Big\},
	\end{align*}
	
	\noindent where ${\rm diag}(x_1,\ldots,x_r) \in M_{rk}(V)$ is the $r \times r$ block diagonal matrix with entries $x_1, \ldots , x_r$ down its diagonal, and the infimum is taken over all such decompositions of $X$.
\end{thm}
\begin{proof}
	The ``$\leq$'' inequality follows simply from the axioms of an operator space: if $X = (X_{ij}) = A\cdot{\rm diag}(x_1,\ldots,x_r)\cdot B^* \in M_m(V)$ then
	\begin{align*}
		(\Phi(X_{ij})) = A\cdot{\rm diag}((id_k \otimes \Phi)(x_1),\ldots,(id_k \otimes \Phi)(x_r))\cdot B^*.
	\end{align*}
	
	\noindent Thus
	\begin{align*}
		\big\|(\Phi(X_{ij}))\big\| \leq \big\|A\big\| \big\|B\big\| \max\big\{\|(id_k \otimes \Phi)(x_1)\|, \ldots, \|(id_k \otimes \Phi)(x_r)\|\big\}.
	\end{align*}
	
	\noindent By taking the supremum over maps $\Phi$ with $\|id_k \otimes \Phi\| \leq 1$, the ``$\leq$'' inequality follows.
	
	We will now show that the infimum on the right is an $L^{\infty}$ matrix norm that coincides with $\|\cdot\|_{M_m(V)}$ for $1 \leq m \leq k$. The ``$\geq$'' inequality will then follow from the fact that $\|\cdot\|_{M_m(MAX^k(V))}$ is the maximal such norm.
	
	First, denote the infimum on the right by $\big\|X\big\|_{m,inf}$ and fix some $1 \leq m \leq k$. Then the inequality $\big\|X\big\|_{M_m(V)} \leq \big\|X\big\|_{m,inf}$ follows immediately by picking any particular decomposition $X = A\cdot{\rm diag}(x_1,\ldots,x_r)\cdot B^*$ and using the axioms of an operator space to see that
	\begin{eqnarray*}
		\big\| A\cdot{\rm diag}(x_1,\ldots,x_r)\cdot B^* \big\|_{M_m(V)} &\leq & \big\|A\big\| \big\|B\big\| \max\big\{\|x_1\|, \ldots, \|x_r\|\big\} \\ &\leq & \big\|A\big\| \big\|B\big\| \\ &\leq & \big\|X\big\|_{m,inf}.
	\end{eqnarray*}
	
	\noindent The fact that equality is attained by some decomposition of $X$ comes simply from writing $X = \big(\big\| X \big\|_{M_m(V)}I\big) \cdot (X \oplus 0_{k-m}) \cdot I$. It follows that $\|\cdot\|_{M_m(V)} = \|\cdot\|_{m,inf}$ for $1 \leq m \leq k$.
	
	All that remains to be proved is that $\|\cdot\|_{m,inf}$ is an $L^\infty$ matrix norm, which we omit as it is directly analogous to the proof of \cite[Theorem 14.2]{Paulsentext}.
\end{proof}

As one final note, observe that we can obtain lower bounds of the $k$-minimal and $k$-maximal operator space norms simply by choosing particular maps $\Phi$ that satisfy the normalization condition of their definition. Upper bounds of the $k$-maximal norms can be obtained from Theorem~\ref{thm:kMaxNormChar}. The problem of computing upper bounds for the $k$-minimal norms was investigated in \cite{JK10b}.

\section{$k$-Super Minimal and $k$-Super Maximal Operator Systems}\label{sec:kMinOpSys}

We will now introduce (abstract) operator systems, and in particular the minimal and maximal operator systems that were explored in \cite{PTT09} and the $k$-super minimal and $k$-super maximal operator systems that were explored in \cite{XThesis}. Our introduction to general operator systems will be brief, and the interested reader is directed to \cite[Chapter 13]{Paulsentext} for a more thorough treatment.

Let $V$ be a complex (not necessarily normed) vector space as before, with a conjugate linear involution that will be denoted by ${}^*$ (such a space is called a \emph{$*$-vector space}). Define $V_h := \{ v \in V : v = v^* \}$ to be the set of Hermitian elements of $V$. We will say that $(V,V^+)$ is an \emph{ordered $*$-vector space} if $V^+ \subseteq V_h$ is a convex cone satisfying $V^+ \cap -V^+ = \{ 0 \}$. Here $V^+$ plays the role of the ``positive'' elements of $V$ -- in the most familiar ordering on square matrices, $V^+$ is the set of positive semidefinite matrices.

Much as was the case with operator spaces, an operator system is constructed by considering the spaces $M_m(V)$, but instead of considering various norms on these spaces that behave well with the norm on $V$, we will consider various cones of positive elements on $M_m(V)$ that behave well with the cone of positive elements $V_+$. To this end, given a $*$-vector space $V$ we let $M_m(V)_h$ denote the set of Hermitian elements in $M_m(V)$. It is said that a family of cones $C_m \subseteq M_m(V)_h$ ($m \geq 1$) is a \emph{matrix ordering} on $V$ if $C_1 = V^+$ and they satisfy the following three properties:
\begin{itemize}
	\item each $C_m$ is a cone in $M_m(V)_h$;
	\item $C_m \cap -C_m = \{0\}$ for each $m$; and
	\item for each $n,m \in \bb{N}$ and $X \in M_{m,n}$ we have $X^* C_m X \subseteq C_n$.
\end{itemize}

A final technical restriction on $V$ is that we will require an element $e \in V_h$ such that, for any $v \in V$, there exists $r > 0$ such that $re - v \in V^+$ (such an element $e$ is called an \emph{order unit}). It is said that $e$ is an \emph{Archimedean order unit} if $re + v \in V^+$ for all $r > 0$ implies that $v \in V^+$. A triple $(V,C_1,e)$, where $(V,C_1)$ is an ordered $*$-vector space and $e$ is an Archimedean order unit, will be referred to as an \emph{Archimedean ordered $*$-vector space} or an \emph{AOU space} for short. Furthermore, if $e \in V_h$ is an Archimedean order unit then we say that it is an \emph{Archimedean matrix order unit} if the operator $e_m := I_m \otimes e \in M_m(V)$ is an Archimedean order unit in $C_m$ for all $m \geq 1$. We are now able to define abstract operator systems:
\begin{defn}
	An \emph{(abstract) operator system} is a triple $(V,\{C_m\}_{m=1}^\infty,e)$, where $V$ is a $*$-vector space, $\{C_m\}_{m=1}^\infty$ is a matrix ordering on $V$, and $e \in V_h$ is an Archimedean matrix order unit.
\end{defn}

For brevity, we may simply say that $V$ is an operator system, with the understanding that there is an associated matrix ordering $\{C_m\}_{m=1}^\infty$ and Archimedean matrix order unit $e$. Recall from \cite{PTT09} that for any AOU space $(V,V^+,e)$ there exists minimal and maximal operator system structure $OMIN(V)$ and $OMAX(V)$ -- that is, there exist particular families of cones $\{C_m^{min}\}_{m=1}^\infty$ and $\{C_m^{max}\}_{m=1}^\infty$ such that if $\{D_m\}_{m=1}^\infty$ is any other matrix ordering on $(V,V^+,e)$ then $C_m^{max} \subseteq D_m \subseteq C_m^{min}$ for all $m \geq 1$. In \cite{XThesis} a generalization of these operator system structures, analogous to the $k$-minimal and $k$-maximal operator spaces presented in Section~\ref{sec:op_space}, was introduced. Given an operator system $V$, the \emph{$k$-super minimal operator system} of $V$ and the \emph{$k$-super maximal operator system} of $V$, denoted $OMIN^k(V)$ and $OMAX^k(V)$ respectively, are defined via the following families of cones:
\begin{align*}
	C_m^{min,k} & := \big\{ (X_{ij}) \in M_m(V) : (\Phi(X_{ij})) \in M_m^{+} \ \ \forall \, \text{unital CP maps } \Phi : V \rightarrow M_k \big\},	\\
	C_m^{max,k} & := \big\{ A\cdot D\cdot A^* \in M_m(V) : A \in M_{m,rk}, D = {\rm diag}(D_1, \ldots, D_r), \\
	 & \quad \quad \quad \quad \quad \quad \quad \quad \quad \quad \quad \ \ D_\ell \in M_k(V)^+ \ \forall \, \ell, r \in \bb{N} \big\}.
\end{align*}

\noindent If $V$ is infinite-dimensional then the cones $C_m^{max,k}$ need not define an operator system due to $I_m \otimes e$ perhaps not always being an Archimedean order unit, though it was shown in \cite{XThesis} how to Archimedeanize the space to correct this problem. We will avoid this technicality by working explicitly in the $V = M_n$ case from now on.

Observe that the interpretation of the $k$-super minimal and $k$-super maximal operator systems is completely analogous to the interpretation of $k$-minimal and $k$-maximal operator spaces. The families of positive cones $C_m^{min,k}$ and $C_m^{max,k}$ coincide with the families of positive cones of $V$ for $1 \leq m \leq k$, and out of all operator system structures on $V$ with this property they are the largest (smallest, respectively) for $m > k$.

In terms of quantum information theory, the cones $C_m^{min,k} \subseteq M_m \otimes M_n$ are exactly the cones of $k$-block positive operators, and the cones $C_m^{max,k} \subseteq M_m \otimes M_n$ are exactly the cones of (unnormalized) density operators $\rho$ with $SN(\rho) \leq k$. These facts have appeared implicitly in the past, but their importance merits making the details explicit:
\begin{thm}\label{thm:opSysConeChar}
	Let $X,\rho \in M_m \otimes M_n$. Then
	\begin{enumerate}[(a)]
		\item $X \in C_m^{min,k}$ if and only if $X$ is $k$-block positive; and
		\item $\rho \in C_m^{max,k}$ if and only if $SN(\rho) \leq k$.
	\end{enumerate}
\end{thm}
\begin{proof}
	To see (a), we will use techniques similar to those used in the proof of Theorem~\ref{thm:MINkChar}. Use the Choi-Kraus representation of completely positive maps so that $X \in C_m^{min,k}$ if and only if
	\begin{align*}
		\sum_{i=1}^{nk} (I_m \otimes A_i)X(I_m \otimes A_i^*) \in (M_m \otimes M_k)^+ \text{ for all } \big\{A_i\big\} \subset M_{k,n} \text{ with } \sum_{i=1}^{nk}A_i A_i^* = I_k.
	\end{align*}

	Now define $\alpha_{ij}\ket{a_{ij}} := A_i^*\ket{j}$ and let $\ket{v} = \sum_{j=0}^{k-1} \gamma_j \ket{c_j} \otimes \ket{j} \in \bb{C}^m \otimes \bb{C}^k$ be an arbitrary unit vector. Then some algebra reveals
	\begin{align*}
		\nu_i\ket{v_i} & := (I_m \otimes A_i^*)\ket{v} = \sum_{j=0}^{k-1} \alpha_{ij} \gamma_j \ket{c_j} \otimes \ket{a_{ij}}.
	\end{align*}
	
	\noindent In particular, $SR(\ket{v_i}) \leq k$ for all $i$. Thus we can write
	\begin{align}\label{eq:cbschmidtSys}
		\sum_{i=1}^{nk}\bra{v}(I_m \otimes A_i)(X)(I_m \otimes A_i^*)\ket{v} = \sum_{i=1}^{nk} \nu_i^2 \bra{v_i} X \ket{v_i} \geq 0.
	\end{align}
	
	\noindent Part (a) follows by noting that we can choose $\ket{v}$ and a CP map with one Kraus operator $A_1$ so that $(I_m \otimes A_1^*)\ket{v}$ is any particular vector of our choosing with Schmidt rank no larger than $k$.
	
	To see the ``only if'' implication of (b), we could invoke various known duality results from operator theory and quantum information theory so that the result would follow from (a), but for completeness we will instead prove it using elementary means. To this end, suppose $\rho \in C_m^{max,k}$. Thus we can write $\rho = A\cdot D\cdot A^*$ for some $A \in M_{m,rk}$ and $D = {\rm diag}(D_1,\ldots,D_r) = \sum_{\ell=1}^r \ketbra{\ell}{\ell} \otimes D_\ell$ with $D_\ell \in M_k(M_n)^{+}$ for all $\ell$. Furthermore, write $D_\ell = \sum_h d_{\ell,h}\ketbra{v_{\ell,h}}{v_{\ell,h}}$ where $\ket{v_{\ell,h}} = \sum_{i=1}^k \ket{i} \otimes \ket{d_{\ell,h,i}}$. Then if we define $\alpha_{\ell,i}\ket{a_\ell,i} := A(\ket{\ell} \otimes \ket{i})$, we have
	\begin{align*}
		A\cdot D\cdot A^* & = \sum_{h=1}^{kn}\sum_{\ell=1}^r d_{\ell,h} \sum_{ij=1}^k A(\ketbra{\ell}{\ell} \otimes \ketbra{i}{j}) A^* \otimes \ketbra{d_{\ell,h,i}}{d_{\ell,h,j}} \\
		& = \sum_{h=1}^{kn}\sum_{\ell=1}^r d_{\ell,h} \sum_{ij=1}^k \alpha_{\ell,i}\alpha_{\ell,j}\ketbra{a_{\ell,i}}{a_{\ell,j}} \otimes \ketbra{d_{\ell,h,i}}{d_{\ell,h,j}} \\
		& = \sum_{h=1}^{kn}\sum_{\ell=1}^r d_{\ell,h} \ketbra{w_{\ell,h}}{w_{\ell,h}},
	\end{align*}
	
	\noindent where
	\begin{align*}
		\ket{w_{\ell,h}} := \sum_{i=1}^k \alpha_{\ell,i}\ket{a_{\ell,i}} \otimes \ket{d_{\ell,h,i}}.
	\end{align*}
	
	\noindent Since $SR(\ket{w_{\ell,h}}) \leq k$ for all $\ell,h$, it follows that $SN(\rho) \leq k$ as well.
	
	For the ``if'' implication, we  note that the above argument can easily be reversed.
\end{proof}

One of the useful consequences of Theorem~\ref{thm:opSysConeChar} is that we can now easily characterize completely positive maps between these various operator system structures. Recall that a map $\Phi$ between operator systems $(V,\{C_m\}_{m=1}^\infty,e)$ and $(V,\{D_m\}_{m=1}^\infty,e)$ is said to be completely positive if $(\Phi(X_{ij})) \in D_m$ whenever $(X_{ij}) \in C_m$. We then have the following result that characterizes $k$-positive maps, entanglement-breaking maps, and $k$-partially entanglement breaking maps as completely positive maps between these $k$-super minimal and $k$-super maximal operator systems.
\begin{cor}\label{cor:kEntangleBreak}
	Let $\Phi : M_n \rightarrow M_n$ and let $k \leq n$. Then
	\begin{enumerate}[(a)]
		\item $\Phi : OMIN^k(M_n) \rightarrow M_n$ is completely positive if and only if $\Phi$ is $k$-partially entanglement breaking;
		\item $\Phi : M_n \rightarrow OMAX^k(M_n)$ is completely positive if and only if $\Phi$ is $k$-partially entanglement breaking;
		\item $\Phi : OMAX^k(M_n) \rightarrow M_n$ is completely positive if and only if $\Phi$ is $k$-positive;
		\item $\Phi : M_n \rightarrow OMIN^k(M_n)$ is completely positive if and only if $\Phi$ is $k$-positive.
	\end{enumerate}
\end{cor}
\begin{proof}
	Fact (a) follows from \cite[Theorem 2]{CK06} and fact (b) follows from the fact that $\Phi$ is $k$-partially entanglement breaking (by definition) if and only if $SN((id_m \otimes \Phi)(\rho)) \leq k$ for all $m \geq 1$. Fact (c) follows from \cite[Theorem 1]{TH00} and (d) follows from (c) and the fact that the cone of unnormalized states with Schmidt number at most $k$ and the cone of $k$-block positive operators are dual to each other \cite{SSZ09}.
\end{proof}

\begin{remark}
Corollary~\ref{cor:kEntangleBreak} was originally proved in the $k = 1$ case in \cite{PTT09} and for arbitrary $k$ in \cite{XThesis}. Both of those proofs prove the result directly, without characterizing the cones $C_m^{min,k}$ and $C_m^{max,k}$ as in Theorem~\ref{thm:opSysConeChar}.
\end{remark}

\section{Norms on Operator Systems}\label{sec:OpSysNorms}

Given an operator system $V$, the \emph{matrix norm induced by the matrix order} $\big\{C_m\big\}_{m=1}^\infty$ is defined for $X \in M_m(V)$ to be
\begin{align}\label{eq:matrixNorm}
	\big\|X\big\|_{M_m(V)} := \inf\left\{ r : \begin{pmatrix}r e_m & X \\ X^* & r e_m \end{pmatrix} \in C_{2m} \right\}.
\end{align}

\noindent In the particular case of $X \in M_m(OMIN^k(V))$ or $X \in M_m(OMAX^k(V))$, we will denote the norm~\eqref{eq:matrixNorm} by $\big\|X\big\|_{M_m(OMIN^k(V))}$ and $\big\|X\big\|_{M_m(OMAX^k(V))}$, respectively. Our first result characterizes $\big\|X\big\|_{M_m(OMIN^k(M_n))}$ in terms of the Schmidt rank of pure states, much like Theorem~\ref{thm:MINkChar} characterized $\big\|X\big\|_{M_m(MIN^k(M_n))}$.
\begin{thm}\label{thm:k_order_norm}
	Let $X \in M_m(OMIN^k(M_n))$. Then
	\begin{align*}
		\big\|X\big\|_{M_m(OMIN^k(M_n))} = \sup_{\ket{v},\ket{w}}\big\{ | \bra{v} X \ket{w} | : & \ SR(\ket{v}),SR(\ket{w}) \leq k \text{ and} \\
			& \ \exists \, P \in M_m \text{ s.t. } (P \otimes I_n)\ket{v} = \ket{w} \big\}.
	\end{align*}
\end{thm}
\begin{proof}
	Given $X \in M_m(OMIN^k(M_n))$, consider the operator
	\begin{align*}
		\tilde{X} := \begin{pmatrix}r I_n & X \\ X^* & r I_n \end{pmatrix} \in (M_2 \otimes M_m) \otimes M_n \cong M_{2m} \otimes M_n.
	\end{align*}
	
	\noindent Then $\tilde{X} \in C^{min,k}_{2m}(M_n)$ if and only if $\bra{v}\tilde{X}\ket{v} \geq 0$ for all $\ket{v} \in \bb{C}^{2m} \otimes \bb{C}^n$ with $SR(\ket{v}) \leq k$. If we multiply on the left and the right by a Schmidt-rank $k$ vector $\ket{v} := \sum_{i=1}^k \beta_i \ket{a_i} \otimes \ket{b_i}$, where $\ket{a_i} = \alpha_{i1}\ket{1} \otimes \ket{a_{i1}} + \alpha_{i2}\ket{2} \otimes \ket{a_{i2}} \in \bb{C}^2 \otimes \bb{C}^m$ and $\ket{b_i} \in \bb{C}^n$, we get
	\begin{align*}
		\bra{v}\tilde{X}\ket{v} & = \sum_{i=1}^k r\big(\beta_i^2\alpha_{i1}^2 + \beta_i^2\alpha_{i2}^2\big) + \sum_{ij=1}^k 2 \alpha_{i1}\alpha_{j2}\beta_i \beta_j{\rm Re}\big((\bra{a_{i1}} \otimes \bra{b_i})X(\ket{a_{j2}} \otimes \ket{b_j})\big)	\\
		& = r + \sum_{ij=1}^k 2\alpha_{i1}\alpha_{j2}\beta_i \beta_j{\rm Re}\big((\bra{a_{i1}} \otimes \bra{b_i})X(\ket{a_{j2}} \otimes \ket{b_j})\big) \\
		& = r + 2 c_1 c_2 {\rm Re}\big(\bra{v_1}X\ket{v_2}\big),
	\end{align*}
	
	\noindent where $c_1\ket{v_1} := \sum_{i=1}^k \alpha_{i1} \beta_i\ket{a_{i1}} \otimes \ket{b_i},c_2\ket{v_2} := \sum_{i=1}^k \alpha_{i2} \beta_i\ket{a_{i2}} \otimes \ket{b_i} \in \bb{C}^m \otimes \bb{C}^n$. Notice that the normalization of the Schmidt coefficients tells us that $c_1^2 + c_2^2 = 1$. Also notice that $\ket{v_1}$ and $\ket{v_2}$ can be written in this way using the same vectors $\ket{b_i}$ on the second subsystem if and only if there exists $P \in M_m$ such that $(P \otimes I_n)\ket{v_1} = \ket{v_2}$. Now taking the infimum over $r$ and requiring that the result be non-negative tells us that the quantity we are interested in is
	\begin{align*}
		\big\|X\big\|_{M_m(OMIN^k(M_n))} & = \sup\Big\{2c_1 c_2{\rm Re}\big(\bra{v_1}X\ket{v_2}\big) : SR(\ket{v_1}),SR(\ket{v_2}) \leq k, c_1^2 + c_2^2 = 1, \\
		& \quad \quad \quad \quad \quad \quad \quad \quad \quad \quad \quad \quad \quad \quad \exists \, P \in M_m \text{ s.t. } (P \otimes I_n)\ket{v_1} = \ket{v_2} \Big\} \\
		& = \sup\Big\{\big|\bra{v_1}X\ket{v_2}\big| : SR(\ket{v_1}),SR(\ket{v_2}) \leq k \text{ and} \\
		& \quad \quad \quad \quad \quad \quad \quad \quad \quad \quad \quad \quad \quad \quad \exists \, P \in M_m \text{ s.t. } (P \otimes I_n)\ket{v_1} = \ket{v_2} \Big\},
	\end{align*}
	
	\noindent where the final equality comes applying a complex phase to $\ket{v_1}$ so that $Re(\bra{v_1}X\ket{v_2}) = |\bra{v_1}X\ket{v_2}|$, and from H\"{o}lder's inequality telling us that the supremum is attained when $c_1 = c_2 = 1/\sqrt{2}$.
\end{proof}

Of course, the matrix norm induced by the matrix order is not the only way to define a norm on the various levels of the operator system $V$. What is referred to as the \emph{order norm} of $v \in V_h$ \cite{PT09} is defined via
\begin{align}\label{eq:orderNorm}
	\|v\|_{or} := \inf\{ t \in \mathbb{R} : -te \leq v \leq te \}.
\end{align}

\noindent It is not difficult to see that for a Hermitian element $X = X^* \in M_m(V)$, the matrix norm induced by the matrix order~\eqref{eq:matrixNorm} coincides with the order norm~\eqref{eq:orderNorm}. It was shown in \cite{PT09} how the order norm on $M_m(V)_h$ can be extended (non-uniquely) to a norm on all of $M_m(V)$. Furthermore, there exists a minimal order norm $\|\cdot\|_m$ and a maximal order norm $\|\cdot\|_M$ satisfying $\|\cdot\|_m \leq \|\cdot\|_M \leq 2\|\cdot\|_m$. We will now examine properties of these two norms as well as some other norms (all of which coincide with the order norm on Hermitian elements) on the $k$-super minimal operator system structures.

We will consider an operator $X \in M_m(OMIN^k(M_n))$, where recall by
this we mean $X \in M_m(M_n)$, where the operator system structure on
the space is $OMIN^k(M_n)$. Then we recall the minimal order norm,
decomposition norm $\|\cdot\|_{dec}$, and maximal order norm from
\cite{PT09}:
\begin{align*}
	\big\|X\big\|_m & := \sup \big\{ |f(X)| : f : M_m(OMIN^k(M_n)) \rightarrow \bb{C} \text{ a pos. linear functional s.t. } f(I) = 1 \big\}, \\
	\big\|X\big\|_{dec} & := \inf \left\{ \bigg\| \sum_{i=1}^r |\lambda_i| P_i \bigg\|_{or} : X = \sum_{i=1}^r \lambda_i P_i \text{ with } P_i \in C^{min,k}_m(M_n) \text{ and } \lambda_i \in \mathbb{C} \right\}, \\
	\big\|X\big\|_M & := \inf \left\{ \sum_{i=1}^r |\lambda_i| \big\| H_i \big\|_{or} : X = \sum_{i=1}^r \lambda_i H_i \text{ with } H_i = H_i^* \text{ and } \lambda_i \in \mathbb{C} \right\}.
\end{align*}

\noindent Our next result shows that the minimal order norm can be thought of in terms of vectors with Schmidt rank no greater than $k$, much like the norms $\|\cdot\|_{S(k)}$ and $\|\cdot\|_{M_m(OMIN^k(M_n))}$ introduced earlier.
\begin{thm}\label{thm:charMinNorm}
	Let $X \in M_m(OMIN^k(M_n))$. Then
	\begin{align*}
		\big\|X\big\|_{m} & = \sup_{\ket{v}}\Big\{ \big| \bra{v} X \ket{v} \big| : SR(\ket{v}) \leq k \Big\}.
	\end{align*}
\end{thm}
\begin{proof}
	Note that if we define a linear functional $f : M_m(OMIN^k(M_n)) \rightarrow \bb{C}$ by $f(X) = \bra{v} X \ket{v}$ for some fixed $\ket{v}$ with $SR(\ket{v}) \leq k$ then it is clear that $f(X) \geq 0$ whenever $X \in C^{min,k}_{m}(M_n)$ (by definition of $k$-block positivity) and $f(I) = 1$. The ``$\geq$'' inequality follows immediately.
	
	To see the other inequality, note that if $X = X^* = (X_{ij})$ and $\ket{v},\ket{w}$ can be written $\ket{v} = \sum_{r=1}^k \alpha_r \ket{a_r} \otimes \ket{b_r}$ and $\ket{w} = \sum_{r=1}^k \gamma_r \ket{c_r} \otimes \ket{b_r}$, then
	\begin{align*}
		\bra{v}X\ket{w} & = \sum_{rs=1}^k \alpha_r \gamma_s \bra{a_r} (\bra{b_r}X_{ij}\ket{b_s})_{ij} \ket{c_s} \\
		 & = (\alpha_1 \bra{a_1}, \cdots, \alpha_k \bra{a_k})\begin{pmatrix}(\bra{b_1}X_{ij}\ket{b_1})_{ij} & \cdots & \bra{b_1}X_{ij}\ket{b_k})_{ij} \\ \vdots & \ddots & \vdots \\ \bra{b_k}X_{ij}\ket{b_1})_{ij} & \cdots & \bra{b_k}X_{ij}\ket{b_k})_{ij}\end{pmatrix}\begin{pmatrix}\gamma_1\ket{c_1} \\ \vdots \\ \gamma_k \ket{c_k}\end{pmatrix}.
	\end{align*}
	
	\noindent Because $X$ is Hermitian, so is the operator in the last line above, so if we take the supremum over all $\ket{v},\ket{w}$ of this form, we may choose $\alpha_i\ket{a_i} = \gamma_i\ket{c_i}$ for all $i$. It follows that
	\begin{align*}
		\sup\Big\{|\bra{v}X\ket{v}| : SR(\ket{v}) \leq k \Big\} = \sup\Big\{\big|\bra{v}X\ket{w}\big| : & \ SR(\ket{v}),SR(\ket{w}) \leq k \text{ and} \\
		& \ \exists \, P \in M_m \text{ s.t. } (P \otimes I_n)\ket{v} = \ket{w} \Big\}.
	\end{align*}
	
	The ``$\leq$'' inequality follows from Theorem~\ref{thm:k_order_norm}, the fact that $\|\cdot\|_{M_m(OMIN^k(M_n))}$ is an order norm, and the minimality of $\|\cdot\|_m$ among order norms.
\end{proof}

The characterization of $\|\cdot\|_m$ given by Theorem~\ref{thm:charMinNorm} can be thought of as in the same vein as \cite[Proposition 5.8]{PT09}, where it was shown that for a unital C$^*$-algebra, $\|\cdot\|_m$ coincides with the numerical radius. In our setting, $\|\cdot\|_m$ can be thought of as a bipartite analogue of the numerical radius, which has been studied in quantum information theory in the $k = 1$ case \cite{GPMSCZ09}.

In the case where $X$ is not Hermitian, equality need not hold between any of the order norms that have been introduced. We now briefly investigate how they compare to each other in general.
\begin{prop}\label{prop:orderNormIneq}
	Let $X \in M_m(OMIN^k(M_n))$. Then
	\begin{align*}
		\big\|X\big\|_m \leq \big\|X\big\|_{M_m(OMIN^k(M_n))} \leq \big\|X\big\|_{dec} \leq \big\|X\big\|_M.
	\end{align*}
\end{prop}
\begin{proof}
	The first and last inequalities follow from the fact that $\|\cdot\|_m$ and $\|\cdot\|_M$ are the minimal and maximal order norms, respectively. Thus, all that needs to be shown is that $\big\|X\big\|_{M_m(OMIN^k(M_n))} \leq \big\|X\big\|_{dec}$. To this end, let $\ket{v},\ket{w} \in \bb{C}^m \otimes \bb{C}^n$ with $SR(\ket{v}),SR(\ket{w}) \leq k$ be such that there exists some $P \in M_m$ such that $(P \otimes I_n)\ket{v} = \ket{w}$. Then for any decomposition $X = \sum_{i=1}^r \lambda_i P_i$ with $P_i \in C^{min,k}_m(M_n)$ and $\lambda_i \in \mathbb{C}$ we can use a similar argument to that used in the proof of Theorem~\ref{thm:charMinNorm} to see that $\bra{v}P_i\ket{v} \geq \big|\bra{v}P_i\ket{w}\big| \geq 0$ because each $P_i$ is $k$-block positive (by Theorem~\ref{thm:opSysConeChar}). Thus
	\begin{align*}
		\big|\bra{v}X\ket{w}\big| = \bigg|\sum_{i=1}^r \lambda_i \bra{v}P_i\ket{w}\bigg| \leq \sum_{i=1}^r |\lambda_i| \big|\bra{v}P_i\ket{w}\big| \leq \sum_{i=1}^r |\lambda_i| \bra{v}P_i\ket{v} \leq \bigg\|\sum_{i=1}^r |\lambda_i| P_i \bigg\|_{or}.
	\end{align*}
	
	\noindent Taking the supremum over all such vectors $\ket{v}$ and $\ket{w}$ and the infimum over all such decompositions of $X$ gives the result.
\end{proof}

We know in general that $\|\cdot\|_m$ and $\|\cdot\|_M$ can differ by at most a factor of two. We now present an example some of these norms and to demonstrate that in fact even $\|\cdot\|_m$ and $\|\cdot\|_{M_m(OMIN^k(M_n))}$ can differ by a factor of two.
\begin{exam}\label{ex:min_op_compare}{\rm
	Consider the rank-$1$ operator $X := \ketbra{\phi}{\psi} \in OMIN^k_n(M_n)$, where
	\begin{align*}
		\ket{\phi} := \frac{1}{\sqrt{n}}\sum_{i=0}^{n-1} \ket{i} \otimes \ket{i} \quad \quad \ket{\psi} := \frac{1}{\sqrt{n}}\sum_{i=0}^{n-1} \ket{i} \otimes \ket{i + 1 (\text{mod } n)}.
	\end{align*}
	
	\noindent It is easily verified that if $\ket{v} = \sum_{i=1}^k \alpha_i \ket{a_i} \otimes \ket{b_i}$ then
	\begin{align*}
		\big|\bra{v}X\ket{v}\big| & = \frac{1}{n} \bigg|\sum_{rs=1}^k \sum_{ij=0}^{n-1} \alpha_r \alpha_s \braket{a_r}{i} \braket{b_r}{i} \braket{j}{a_s} \braket{j + 1 (\text{mod } n)}{b_s}\bigg| \\
		& = \frac{1}{n} \bigg| \Tr\Big(\sum_{r=1}^k \alpha_r \overline{\ket{a_r}}\bra{b_r} \Big) \cdot \sum_{j=0}^{n-1} \bra{j} \Big(\sum_{r=1}^k \alpha_r \overline{\ket{a_r}}\bra{b_r} \Big) \ket{j + 1 (\text{mod } n)} \bigg|.
	\end{align*}
	
	\noindent In the final line above we have the trace of an operator with rank at most $k$, multiplied by the sum of the elements on the superdiagonal of the same operator, subject to the constraint that the Frobenius norm of that operator is equal to $1$. It follows that $\big|\bra{v}X\ket{v}\big| \leq \frac{k}{2n}$ and so $\|X\|_m = \frac{k}{2n}$ (equality can be seen by taking $\ket{v} = \sum_{i=0}^{k-1} \frac{1}{\sqrt{2k}}\ket{i} \otimes (\ket{i} + \ket{i + 1 (\text{mod } n)})$).
	
	To see that $\|X\|_{op}$ is twice as large, consider $\ket{v} = \frac{1}{\sqrt{k}}\sum_{i=0}^{k-1} \ket{i} \otimes \ket{i}$ and $\ket{w} = \frac{1}{\sqrt{k}}\sum_{i=0}^{k-1} \ket{i} \otimes \ket{i + 1 (\text{mod } n)}$. Then it is easily verified that $\bra{v}X\ket{w} = \frac{k}{n}$. Moreover, if $P \in M_m$ is the cyclic permutation matrix such that $P\ket{i} = \ket{i - 1 (\text{mod } n)}$ for all $i$ then $(P \otimes I_n)\ket{v} = \ket{w}$, showing that $\|X\|_{op} \geq \frac{k}{n}$.
}\end{exam}

\section{Contractive Maps as Separability Criteria}\label{sec:CBNorm}

We now investigate the completely bounded version of the $k$-minimal operator space norms and $k$-super minimal operator system norms that have been introduced. We will see that these completely bounded norms can be used to provide a characterization of Schmidt number analogous to its characterization in terms of $k$-positive maps.

Given operator spaces $V$ and $W$, the completely bounded (CB) norm from $V$ to $W$ is defined by
\begin{align*}
	\big\|\Phi\big\|_{CB(V,W)} := \sup_{m \geq 1}\Big\{ \big\|(id_m \otimes \Phi)(X)\big\|_{M_m(W)} : X \in M_m(V) \text{ with } \big\|X\big\|_{M_m(V)} \leq 1 \Big\}.
\end{align*}

\noindent Clearly this reduces to the standard completely bounded norm of $\Phi$ in the case when $V = M_r$ and $W = M_n$. We will now characterize this norm in the case when $V = M_r$ and $W = MIN^k(M_n)$. In particular, we will see that the $k$-minimal completely bounded norm of $\Phi$ is equal to the perhaps more familiar operator norm $\big\|id_k \otimes \Phi\big\|$ -- that is, the CB norm in this case stabilizes in much the same way that the standard CB norm stabilizes (indeed, in the $k = n$ case we get exactly the standard CB norm). This result was originally proved in \cite{OR04}, but we prove it here using elementary means for completeness and clarity, and also because we will subsequently need the operator system version of the result, which can be proved in the same way.

\begin{thm}\label{thm:mainCB}
	Let $\Phi : M_r \rightarrow M_n$ be a linear map and let $1 \leq k \leq n$. Then
	\begin{align*}
		\big\|id_k \otimes \Phi\big\| & = \big\|\Phi\big\|_{CB(M_r,MIN^k(M_n))}.
	\end{align*}
\end{thm}

\begin{proof}
	To see the ``$\leq$'' inequality, simply notice that $\big\|Y\big\|_{M_k(MIN^k(M_n))} = \big\|Y\big\|_{M_k(M_n)}$ for all $Y \in M_k(M_n)$. We thus just need to show the ``$\geq$'' inequality, which we do in much the same manner as Smith's original proof that the standard CB norm stabilizes.
	
	First, use Theorem~\ref{thm:MINkChar} to write
	\begin{align}\label{eq:cb_k_sup}
		\big\|\Phi\big\|_{CB(M_r,MIN^k(M_n))} = \sup_{m \geq 1}\Big\{ \big\|(id_m \otimes \Phi)(X)\big\|_{S(k)} : \big\|X\big\| \leq 1 \Big\}.
	\end{align}
	
	\noindent Now fix $m \geq k$ and a pure state $\ket{v} \in \bb{C}^m \otimes \bb{C}^n$ with $SR(\ket{v}) \leq k$. We begin by showing that there exists an isometry $V : \bb{C}^k \rightarrow \bb{C}^m$ and a state $\ket{\tilde{v}} \in \bb{C}^k \otimes \bb{C}^n$ such that $(V \otimes I_n)\ket{\tilde{v}} = \ket{v}$. To this end, write $\ket{v}$ in its Schmidt Decomposition $\ket{v} = \sum_{i=1}^{k} \alpha_i \ket{a_i} \otimes \ket{b_i}$. Because $k \leq m$, we may define an isometry $V : \bb{C}^k \rightarrow \bb{C}^m$ by $V\ket{i} = \ket{a_i}$ for $1 \leq i \leq k$. If we define $\ket{\tilde{v}} := \sum_{i=1}^{k} \alpha_i \ket{i} \otimes \ket{b_i}$ then $(V \otimes I_n)\ket{\tilde{v}} = \ket{v}$, as desired.
	
	Now choose $\tilde{X} \in M_m(M_r)$ such that $\big\|\tilde{X}\big\| \leq 1$ and the supremum~\eqref{eq:cb_k_sup} (holding $m$ fixed) is attained by $\tilde{X}$. Then choose vectors $\ket{v},\ket{w} \in \bb{C}^m \otimes \bb{C}^n$ with $SR(\ket{v}),SR(\ket{w}) \leq k$ such that
	\begin{align*}
		\big\| (id_m \otimes \Phi)(\tilde{X}) \big\|_{S(k)} = \big|\bra{v} (id_m \otimes \Phi)(\tilde{X}) \ket{w}\big|.
	\end{align*}
	
	\noindent As we saw earlier, there exist isometries $V,W : \bb{C}^k \rightarrow \bb{C}^m$ and unit vectors $\ket{\tilde{v}},\ket{\tilde{w}} \in \bb{C}^k \otimes \bb{C}^n$ such that $(V \otimes I_n)\ket{\tilde{v}} = \ket{v}$ and $(W \otimes I_n)\ket{\tilde{w}} = \ket{w}$. Thus
	\begin{align*}
		\big\| (id_m \otimes \Phi)(\tilde{X}) \big\|_{S(k)} & = \big|\bra{\tilde{v}}(V^* \otimes I_n) (id_m \otimes \Phi)(\tilde{X}) (W \otimes I_n)\ket{\tilde{w}} \big| \\
		& = \big|\bra{\tilde{v}} (id_k \otimes \Phi)((V^* \otimes I_r)\tilde{X}(W \otimes I_r))\ket{\tilde{w}} \big| \\
		& \leq \big\| (id_k \otimes \Phi)((V^* \otimes I_r)\tilde{X}(W \otimes I_r)) \big\| \\
		& \leq \sup \Big\{ \big\| (id_k \otimes \Phi)(X) \big\| : X \in M_k(M_r) \text{ with } \big\|X\big\| \leq 1 \Big\},
	\end{align*}
	
\noindent where the final inequality comes from the fact that $\big\|(V^* \otimes I_r)\tilde{X}(W \otimes I_r)\big\| \leq 1$. The desired inequality follows, completing the proof.
\end{proof}

We will now show that the operator system versions of these norms have applications to testing separability of quantum states. To this end, notice that if we instead consider the completely bounded norm from $M_r$ to the $k$-super minimal operator \emph{systems} on $M_n$, then a statement that is analogous to Theorem~\ref{thm:mainCB} holds. Its proof can be trivially modified to show that if $\Phi : M_r \rightarrow M_n$ and $1 \leq k \leq n$ then
\begin{align}\begin{split}\label{eq:opSysStab}
	& \sup\Big\{ \big|\bra{v}(id_k \otimes \Phi)(X)\ket{v}\big| : \big\|X\big\| \leq 1, X = X^* \Big\} \\
	= & \sup_{m \geq 1}\Big\{ \big|\bra{v}(id_m \otimes \Phi)(X)\ket{v}\big| : \big\|X\big\| \leq 1, X = X^*, SR(\ket{v}) \leq k \Big\}.
\end{split}
\end{align}

Equation~\eqref{eq:opSysStab} can be thought of as a stabilization result for the completely bounded version of the norm described by Theorem~\ref{thm:charMinNorm}. We could also have picked one of the other order norms on the $k$-super minimal operator systems to work with, but from now on we will be working exclusively with Hermiticity-preserving maps $\Phi$. So by the fact that all of the operator system order norms are equal on Hermitian operators, it follows that these versions of their completely bounded norms are all equal as well.

Before proceeding, we will need to define some more notation. If $\Phi : M_n \rightarrow M_r$ is a linear map, then we define a Hermitian version of the induced trace norm of $\Phi$:
\begin{align*}
	\big\|\Phi\big\|_{tr}^H := \sup\Big\{ \big\| \Phi(X) \big\|_{tr} : \big\|X\big\|_{tr} \leq 1, X = X^* \Big\}.
\end{align*}

\noindent Because of convexity of the trace norm, it is clear that the above norm is unchanged if instead of being restricted to Hermitian operators, the supremum is restricted to positive operators or even just projections. Now by taking the dual of the left and right norms described by Equation~\eqref{eq:opSysStab}, and using the fact that the operator norm is dual to the trace norm, we arrive at the following corollary:

\begin{cor}\label{cor:mainCB}
	Let $\Phi : M_n \rightarrow M_r$ be a Hermiticity-preserving linear map and let $1 \leq k \leq n$. Then
	\begin{align*}
		\big\|id_k \otimes \Phi\big\|^{H}_{tr} = \sup_{m \geq 1}\Big\{ \big\|(id_m \otimes \Phi)(\rho)\big\|_{tr} : \rho \in M_m \otimes M_n \text{ with } SN(\rho) \leq k \Big\}.
	\end{align*}
\end{cor}

We will now characterize the Schmidt number of a state $\rho$ in terms of maps that are contractive in the norm described by Corollary~\ref{cor:mainCB}. Our result generalizes the separability test of \cite{HHH06}. We begin with a simple lemma that will get us most of the way to the linear contraction characterization of Schmidt number. The $k = 1$ version of this lemma appeared as \cite[Lemma 1]{HHH06}, though our proof is more straightforward.

\begin{lemma}\label{lem:contrac}
	Let $\rho \in M_m \otimes M_n$ be a density operator. Then $SN(\rho) \leq k$ if and only if $(id_m \otimes \Phi)(\rho) \geq 0$ for all trace-preserving $k$-positive maps $\Phi : M_n \rightarrow M_{2n}$.
\end{lemma}
\begin{proof}
	The ``only if'' implication of the proof is clear, so we only need to establish that if $SN(\rho) > k$ then there is a trace-preserving $k$-positive map $\Phi : M_n \rightarrow M_{2n}$ such that $(id_m \otimes \Phi)(\rho) \not\geq 0$. To this end, let $\Psi : M_n \rightarrow M_n$ be a $k$-positive map such that $(id_m \otimes \Psi)(\rho) \not\geq 0$ (which we know exists by \cite{TH00,RA07}). Without loss of generality, $\Psi$ can be scaled so that $\big\|\Psi\big\|_{tr} \leq \frac{1}{n}$. Then if $\Omega : M_n \rightarrow M_n$ is the completely depolarizing channel defined by $\Omega(\rho) = \frac{1}{n}I_n$ for all $\rho \in M_n$, it follows that $(\Omega - \Psi)(\rho) \geq 0$ for all $\rho \geq 0$ and so the map $\Phi := \Psi \oplus (\Omega - \Psi) : M_n \rightarrow M_{2n}$ is $k$-positive (and easily seen to be trace-preserving). Because $(id_m \otimes \Psi)(\rho) \not\geq 0$, we have $(id_m \otimes \Phi)(\rho) \not\geq 0$ as well, completing the proof.
\end{proof}

We are now in a position to prove the main result of this section. Note that in the $k = 1$ case of the following theorem it is not necessary to restrict attention to Hermiticity-preserving linear maps $\Phi$ (and indeed this restriction was not made in \cite{HHH06}), but our proof for arbitrary $k$ does make use of Hermiticity-preservation.
\begin{thm}
	Let $\rho \in M_m \otimes M_n$ be a density operator. Then $SN(\rho) \leq k$ if and only if $\big\|(id_m \otimes \Phi)(\rho)\big\|_{tr} \leq 1$ for all Hermiticity-preserving linear maps $\Phi : M_n \rightarrow M_{2n}$ with $\big\| id_k \otimes \Phi \big\|_{tr}^{H} \leq 1$.
\end{thm}
\begin{proof}
	To see the ``only if'' implication, simply use Corollary~\ref{cor:mainCB} with $r = 2n$.
	
	For the ``if'' implication, observe that any positive trace-preserving map $\Psi$ is necessarily Hermiticity-preserving and has $\big\|\Psi\big\|_{tr}^H \leq 1$. Letting $\Psi = id_k \otimes \Phi$ then shows that any $k$-positive trace-preserving map $\Phi$ has $\big\|id_k \otimes \Phi\big\|_{tr}^H \leq 1$. Thus the set of Hermiticity-preserving linear maps $\Phi$ with $\big\|id_k \otimes \Phi\big\|_{tr}^H \leq 1$ contains the set of $k$-positive trace-preserving maps, so the ``if'' implication follows from Lemma~\ref{lem:contrac}.
\end{proof}

\vspace{0.1in}

\noindent{\bf Acknowledgements.} Thanks are extended to Marius Junge for drawing our attention to the $k$-minimal and $k$-maximal operator space structures. N.J. was supported by an NSERC Canada Graduate Scholarship and the University of Guelph Brock Scholarship. D.W.K. was supported by Ontario Early Researcher Award 048142, NSERC Discovery Grant 400160 and NSERC Discovery Accelerator Supplement 400233. R.P. was supported by NSERC Discovery Grant 400096.



\begin{thebibliography}{99}

\bibitem{Paulsentext} V. I. Paulsen, \emph{Completely bounded maps and operator algebras}.
Cambridge University Press, Cambridge (2003).

\bibitem{Pistext} G. Pisier, \emph{Introduction to operator space theory}. Cambridge University Press, Cambridge (2003).

\bibitem{HHH09} R. Horodecki, P. Horodecki, M. Horodecki, K. Horodecki, \emph{Quantum entanglement}. Rev. Mod. Phys. {\bf 81}, 865--942 (2009).

\bibitem{BZtext} I. Bengtsson, K. {\.Z}yczkowski, \emph{Geometry of quantum states: an introduction to quantum entanglement}. Cambridge University Press, Cambridge (2006).

\bibitem{HHH96} M. Horodecki, P. Horodecki, R. Horodecki, \emph{Separability of mixed states: necessary and sufficient conditions}. Physics Letters A {\bf 223}, 1--8 (1996).

\bibitem{HHH06} M. Horodecki, P. Horodecki, R. Horodecki, \emph{Separability of mixed quantum states: linear contractions approach}. Open Syst. Inf. Dyn. {\bf 13}, 103 (2006).

\bibitem{TH00} B.~M. Terhal, P. Horodecki, \emph{Schmidt number for density matrices}.  Phys. Rev. A {\bf 61}, 040301R (2000).

\bibitem{RA07} K.~S. Ranade, M. Ali, \emph{The Jamio\l kowski isomorphism and a simplified proof for the correspondence between vectors having Schmidt number $k$ and $k$-positive maps}. Open Syst. Inf. Dyn. {\bf 14}, 371--378 (2007).

\bibitem{HSR03} M. Horodecki, P.~W. Shor, M.~B. Ruskai, \emph{General entanglement breaking channels}. Rev. Math. Phys {\bf 15}, 629--641 (2003).

\bibitem{PTT09} V. Paulsen, I. Todorov, M. Tomforde, \emph{Operator system structures on ordered spaces}. Proc. of the LMS (to appear).

\bibitem{XThesis} B. Xhabli, \emph{Universal operator system structures on ordered spaces and their applications}. PhD Thesis (2009).

\bibitem{CK06} D. Chruscinski, A. Kossakowski, \emph{On partially entanglement breaking channels}. Open Sys. Information Dyn. {\bf 13}, 17--26 (2006).

\bibitem{JK10a} N. Johnston, D. W. Kribs, \emph{A family of norms with applications in quantum information theory}. To appear in J. Math. Phys. (2010). arXiv:0909.3907v3 [quant-ph]

\bibitem{JK10b} N. Johnston, D. W. Kribs, \emph{A family of norms with applications in quantum information theory II}. Preprint (2010). 	 arXiv:1006.0898v1 [quant-ph]

\bibitem{J10} N. Johnston, \emph{Characterizing operations preserving separability measures via linear preserver problems}. Preprint (2010). arXiv:1008.3633v1 [quant-ph]

\bibitem{PPHH10} L. Pankowski, M. Piani, M. Horodecki, P. Horodecki, \emph{A few steps more towards NPT bound entanglement}. To appear in IEEE Trans. Inf. Theory (2010). arXiv:0711.2613v2 [quant-ph]

\bibitem{CKo09} D. Chru\'{s}ci\'{n}ski, A. Kossakowski, \emph{Spectral conditions for positive maps}. Commun. Math. Phys. {\bf 290}, 1051–1064 (2009)

\bibitem{CKS09} D. Chru\'{s}ci\'{n}ski, A. Kossakowski, G. Sarbicki, \emph{Spectral conditions for entanglement witnesses vs. bound entanglement}. Preprint (2009). arXiv:0908.1846v1 [quant-ph]

\bibitem{DSSTT00} D. P. DiVincenzo, P. W. Shor, J. A. Smolin, B. M. Terhal, A. V. Thapliyal, \emph{Evidence for Bound Entangled States with Negative Partial Transpose}, Phys. Rev. A {\bf 61}, 062312 (2000). arXiv:quant-ph/9910026v3

\bibitem{S83} R. R. Smith, \emph{Completely bounded maps between C$^*$-algebras}. J. London Math. Soc. {\bf 27}, 157-166 (1983).

\bibitem{Kit97} A.Yu. Kitaev, {\em Quantum computations: algorithms and error correction,} Russian Math. Surveys  {\bf
52} (1997), 1191-1249.

\bibitem{OR04} T. Oikhburg, E. Ricard, \emph{Operator spaces with few completely bounded maps}. Math. Ann. {\bf 328}, 229–-259 (2004).

\bibitem{JKP09} N. Johnston, D. W. Kribs, and V. Paulsen, \emph{Computing Stabilized Norms for Quantum Operations}. Quantum Information \& Computation {\bf 9} 1 \& 2, 16--35 (2009).

\bibitem{SSZ09} {\L}. Skowronek, E. St{\o}rmer, and K. {\.Z}yczkowski, \emph{Cones of positive maps and their duality relations}. J. Math. Phys. {\bf 50}, 062106 (2009).

\bibitem{PT09} V. Paulsen, M. Tomforde, \emph{Vector spaces with an order unit}. Indiana Univ. Math. J (to appear). arXiv:0712.2613v4 [math.OA]

\bibitem{GPMSCZ09} P. Gawron, Z. Puchala, J.~A. Miszczak, L. Skowronek, M.-D. Choi, K. Zyczkowski, \emph{Local numerical range: a versatile tool in the theory of quantum information}. E-print: arXiv:0905.3646v1 [quant-ph]

\end{thebibliography}
\end{document}